\documentclass[12pt]{amsart}
\usepackage{amsmath,amssymb,epsfig,color}
\newtheorem{theorem}{Theorem}[section]

\newtheorem{corollary}[theorem]{Corollary}
\newtheorem{definition}[theorem]{Definition}
\newtheorem{lemma}[theorem]{Lemma}
\newtheorem{proposition}[theorem]{Proposition}
\newtheorem{remark}[theorem]{Remark}

\newcommand{\vanish}[1]{}\parskip=12pt

\newcommand{\0}{\widehat{0}}
\newcommand{\1}{\widehat{1}}

\begin{document}
\title{A second look at the toric h-polynomial of a cubical complex} 
\author{G\'abor Hetyei}
\address{Department of Mathematics and Statistics, UNC Charlotte, 
	Charlotte, NC 28223}
\email{ghetyei@uncc.edu}
\thanks{This work was supported by the NSA grant
\# H98230-07-1-0073.}
\begin{abstract}
We provide an explicit formula for the toric $h$-contribution of each
cubical shelling component, and a new combinatorial model to prove 
Clara Chan's result on the non-negativity of these contributions. Our
model allows for a variant of the Gessel-Shapiro result on the
$g$-polynomial of the cubical lattice, this  variant may be shown by simple
inclusion-exclusion. We establish an isomorphism between our model and
Chan's model and provide a reinterpretation in terms of noncrossing 
partitions. By discovering another variant of the Gessel-Shapiro result
in the work of Denise and Simion, we find evidence that the toric
$h$-polynomials of cubes are related to the Morgan-Voyce polynomials
via Viennot's combinatorial theory of orthogonal polynomials.   
\end{abstract}
\maketitle

\section*{Introduction}
The toric $h$-polynomial of a lower Eulerian partially
ordered set, introduced by Stanley~\cite[\S 2]{Stanley-GH} 
is a highly sophisticated invariant. As noted by Stanley
himself~\cite[p. 193]{Stanley-GH}, it 
``seems quite difficult to compute,'' ``without using the laborious
defining recurrence''. Yet in those few cases where a combinatorial
interpretation was found, it lead to deep enumerative
results. An excellent example of this is Chan's work~\cite{Chan}
who showed that the toric $h$-polynomial of a shellable cubical 
complex has nonnegative coefficients, by finding a set of highly nontrivial
enumeration problems for which the number of objects counted turns out
to be exactly the contribution of a cubical shelling component to 
a given coefficient in the toric $h$-polynomial. She later extended her
results in collaboration with Billera and Liu~\cite{Billera-Chan-Liu}   
showing that the toric $h$-polynomial of a star-shellable cubical
complex is the $h$-polynomial of a shellable simplicial complex.
It seems that very little has been done in the area ever since. 

The main purpose of this paper is to convince the reader that the 
toric $h$-polynomials of shellable cubical complexes are worth a ``second
look''. There are other combinatorial models for Chan's results, 
some of them linked to her model via non-trivial bijections, others 
may be derived in the future from an unexpected and unstated 
appearance of the toric $g$-polynomial of a cubical lattice in the work
of Denise and Simion~\cite{Denise-Simion}. The explicit formulas 
we find for the contribution
of each shelling component to the toric $h$-polynomial of a shellable
cubical complex link the study of these polynomials to the combinatorics
of Catalan numbers. It is not the first time that such a connection appears
in the literature: the formulas of Bayer and Ehrenborg~\cite[Theorems
  4.1 and 4.2]{Bayer-Ehrenborg}
expressing the toric $h$-polynomial of an Eulerian poset in terms of its
flag $h$-vector and $cd$-index abound in Catalan numbers, and same holds
for the formula of Billera and Brenti~\cite[Theorem 3.3]{Billera-Brenti}
expressing the Kazhdan-Lusztig polynomial of any Bruhat interval in any Coxeter
group in terms of the complete $cd$-index.

Our paper is structured as follows. After the Preliminaries, in
Section~\ref{s_expl} we use Chan's recursions~\cite{Chan}
to build explicit formulas for the contribution
of each shelling component to the toric $h$-polynomial of a shellable
cubical complex. These formulas are meant to serve as ``yardstick'', against
which present or future combinatorial models can be
measured. 

In Section~\ref{s_pte} we introduce a new combinatorial model, counting
the total weight of postorder trees which are weighted according to the
numbers of  ``first-born'' leaf children,  not ``first-born'' leaf children and 
``first-born'' non-leaf children in the tree.  The greatest
simplification occurs when we prove the variant of the Gessel-Shapiro
result~\cite[Ex. 3.71g]{Stanley-EC1} on the toric $g$-polynomial of a
cubical lattice. The original variant counts forks in a plane tree, and
the easiest proof on record involves solving a quadratic equation for a
generating function, whereas in our model we count ``first-born'' leaves 
in plane trees and there is a proof essentially based on
inclusion-exclusion. In our model, all three types of ``special vertices'' are
equally simply defined, whereas the third type of ``special vertices''
in Chan's original model~\cite{Chan} has a little more technical
definition. 

Our model is closely related to Chan's original model~\cite{Chan}, and a
bijection is outlined in Section~\ref{s_reflect}. This calls for
transforming a preorder tree into a binary tree, reflecting the binary
tree in the plane, transforming it back to a plane tree, whose mirror
image is a postorder tree. Under this correspondence forks in the
postorder tree correspond to first-born leaf children of a nonroot
parent, and the other two classes
of special vertices also turn out to be bijective equivalents of the
classes in~\cite{Chan}.

The postorder tree model is easily transformed into a noncrossing
partition model, and the results of Chan~\cite{Chan} have a
rephrasing in terms of noncrossing partitions,
where we need to consider a statistics of nonsingleton blocks,
singleton blocks, and elements that are not the minimum in their block,
nor are they the maximum in a block that does not contain $1$.
This is shown in Section~\ref{s_nc}. We also note
that the Gessel-Shapiro result has a variant in the work of Denise and
Simion~\cite{Denise-Simion}. This observation allows us to present a new
recursion formula for the toric $g$-polynomials of cubical lattices, and 
establish a relation between these polynomials and the Morgan-Voyce
polynomials via Viennot's combinatorial
theory of orthogonal polynomials~\cite{Viennot}. This connection
deserves further study in a future work. 

Hopefully our ``second look'' will not be the last look at these
remarkable polynomials, and in the concluding Section~\ref{s_c}
we collect some of the questions raised by the present work. 

\section*{Acknowledgments}
I wish to thank Mireille Bousquet-M\'elou for acquainting me with
Viennot's work~\cite{Viennot}, Louis Billera for advice and
encouragement, and Clara Chan for valuable
information on the toric $h$-polynomials of cubical complexes.
I am indebted to an anonymous referee and to Russ Woodroofe for
correcting some of my most confusing mistakes regarding history and terminology.
This work was supported by the NSA grant
\# H98230-07-1-0073.

\section{Preliminaries}

\subsection{The toric $h$-polynomial of a lower Eulerian poset and a
  polyhedral complex}

A partially ordered poset is {\em graded}, if 
it has a unique minimum element $\0$, a unique maximum element $\1$ and a 
rank function $\rho$. A graded partially ordered poset is {\em Eulerian} if in
any open interval $(x,y)$ the number of elements at odd ranks equals the
number of elements at even ranks. 
Given an Eulerian poset $\widehat{P}$, let us denote by $P$ the poset
obtained by removing the maximum element $\1$ from $\widehat{P}$.   
Stanley~\cite[\S 2]{Stanley-GH} defines the {\em toric $h$-vector of
$P$} by first defining the polynomials $f(P,x)$ and 
$g(P,x)$ by the following intertwined recurrence:
\begin{itemize}
\item[(a)] $f(\emptyset,x)=g(\emptyset,x)=1$
\item[(b)] If $\widehat{P}$ has rank $d+1\geq 1$ and if $f(P,x)=k_0+k_1
  x+\cdots$ then 
$$
g(P,x)=\sum_{i=0}^{\lfloor d/2\rfloor} (k_i-k_{i-1}) x^i
$$
(we set $k_{-1}=0$),
\item[(c)] If $\widehat{P}$ has rank $d+1\geq 1$ then
$$
f(P,x)=\sum_{t\in P} g([\0,t),x) (x-1)^{d-\rho(t)}.
$$
\end{itemize} 
If $f(P,x)=k_0+k_1 x+\cdots +k_d x^d$, then he sets $h_i:=k_{d-i}$ and
calls the vector $(h_0,\ldots,h_d)$ the {\em $h$-vector} of $P$. Since, 
by \cite[Theorem 2.4]{Stanley-GH}, any Eulerian poset $\widehat{P}$
of rank $d+1$ satisfies $h_i=h_{d-i}$, we may refer to $f(P,x)$ as the
{\em toric $h$-polynomial of $P$}, in which the coefficient of $x^i$ is
$h_i$.

A finite set $P$ is {\em lower Eulerian} if it has a unique
minimum element $\0$, and every interval of the form $[\0,t]$ in it is
Eulerian. Thus the polynomials $f([\0,t),x)$ $g([\0,t),x)$ are still
defined for all $t\in P$. Stanley~\cite[Eq. (18)]{Stanley-GH} extends the
definition of $f(P,x)$ to a lower Eulerian poset $P$ by 
$$
f(P,x)=\sum_{t\in P} g([\0,t), x)(x-1)^{d-\rho(t)},
$$
where $d$ is the maximum chain length in $P$ and $\rho(t)$ is the
length of a maximal chain in $[\0,t)$. 

An important example of a lower Eulerian poset is the face poset
$P({\mathcal P})$  of a {\em polyhedral complex} ${\mathcal P}$, ordered
by inclusion. Given a $d$-dimensional 
polyhedral complex ${\mathcal P}$, for each face $F\in {\mathcal P}$ 
the half open interval $[\emptyset, F)\subset P({\mathcal P})$ is the
  face lattice of the boundary complex $\partial F$ of $F$. Introducing 
$\overline{h}({\mathcal P},x)$ as a new notation for the $h$-polynomial 
$f(P({\mathcal P}),x)$ and $g({\mathcal P}, x)$ as a shorthand for 
for $g(P({\mathcal P}), x)$ we obtain the following adaptation of
Stanley's original definition to polyhedral complexes:
\begin{enumerate}
\item $\overline{h}(\emptyset, x)=g(\emptyset, x)=1$.
\item If $\dim {\mathcal P}=d\geq 0$ then
  $\overline{h}({\mathcal P}, x)=\sum_{F\in {\mathcal P}} g(\partial
  F,x)(x-1)^{d-\dim F}$.
\item $\dim {\mathcal P}=d\geq 0$ then
  $g({\mathcal P},x)=\sum_{i=0}^{\lfloor(d+1)/2\rfloor}
  (k_i-k_{i-1}) x^i$ where $k_i$ is the coefficient of $x^i$ in
  $\overline{h}({\mathcal P}, x)$.  
\end{enumerate}
This adaptation may be found in the paper~\cite{Billera-Chan-Liu} of
Billera, Chan, and Liu. They then define the {\em toric $h$-vector}
$h({\mathcal P})=(h_0,\ldots,h_{d+1})$ of the polyhedral complex
${\mathcal P}$ by the formula  
$$
h(P,x)=\sum_{i=0}^{d+1}h_i x^i=x^{d+1}\overline{h}({\mathcal P}, 1/x),
$$
and the {\em toric $g$-vector} of ${\mathcal P}$ by $g({\mathcal
  P})=(h_0,h_1-h_0,\ldots, h_m-h_{m-1})$ where $m=\lfloor(d+1)/2\rfloor$. 

As noted above, Stanley~\cite[Theorem 2.4]{Stanley-GH} has shown that
for an Eulerian poset $\widehat{P}$ of rank $d+1\geq 1$ the polynomial
$f(P,x)$ satisfies the {\em   generalized Dehn-Sommerville equations} 
\begin{equation}
\label{E_DS}
x^d f(P,1/x)=f(P,x).
\end{equation}
Let now $\widehat{Q}$ be the poset obtained from $\widehat{P}$ by adding
a new maximum element. (Thus $Q=\widehat{P}$.) Using (c) of Stanley's
definition we obtain
$$
f(Q,x)=\sum_{t\in P} g([\0,t),x) (x-1)^{d+1-\rho(t)}+g(P,x)
$$
from which, using (\ref{E_DS}) is not hard to derive the equation
\begin{equation}
\label{E_fg}
x^{d+1} f(Q,1/x)=g(P,x).
\end{equation}
For polyhedral complexes this equation has the following consequence.
\begin{corollary}
\label{C_gh}
Let ${\mathcal P}$ be the complex of all faces of a convex polytope, and
$\partial {\mathcal P}$ its boundary complex. Then we have
$$
h({\mathcal P},x)=g(\partial {\mathcal P},x).
$$ 
\end{corollary}
In particular, the $h$-polynomial of a cube is the $g$-polynomial of its
boundary, as this was noted in~\cite[Lemma 3.1]{Billera-Chan-Liu}. Note
that the $g$-polynomial of the boundary of a $d$-dimensional cube (as a
$g$-polynomial of a polyhedral complex) is the same as the
$g$-polynomial of $L_d$ (as the $g$-polynomial of a poset) where
$\widehat{L}_d$ is the face lattice of a $d$-dimensional cube. 

\subsection{Results on the toric $h$-vector of a cubical complex}

The first result on the toric $g$- and $h$-polynomials of a cubical
complex is due to Gessel and it may be found in
Stanley's seminal paper~\cite[Proposition 2.6]{Stanley-GH}.
This states that for the face lattice $\widehat{L}_d$ of a
$d$-dimensional cube we have 
\begin{equation}
\label{E_gd}
g(L_d,x)=\sum_{k=0}^{\lfloor d/2\rfloor} \frac{1}{d-k+1}
\binom{d}{k}\binom{2d-2k}{d} (x-1)^k. 
\end{equation}
By Corollary~\ref{C_gh} this is also the $h$-polynomial of the
$d$-dimensional cube. A combinatorial interpretation of the right hand
side of (\ref{E_gd}) is due to
Shapiro~\cite[Ex. 3.71g]{Stanley-EC1} stating that the coefficient of
$x^i$ in $g(L_d,x)$ is the number of plane trees on $d+1$ vertices with
$i$ {\em forks}, where a fork is a vertex having  more than one child. A
proof that involves solving a quadratic equation for a generating
function is given in~\cite[Proposition 2]{Chan}.  

A complete description of the $h$-polynomial of a {\em shellable}
cubical complex was given by Clara Chan~\cite{Chan}. The general notion of
shellability was used by topologists for a long time. The fact that the
boundary complex of any convex polytope is shellable was shown by Bruggeser and
Mani~\cite{Bruggeser-Mani}. A $(d-1)$-dimensional cubical complex is
shellable exactly when there is an ordering $(F_1,\ldots,F_r)$ of its facets 
such that for $t>1$ the intersection $F_t\cap (F_1\cap\cdots\cap F_{t-1})$
is a union of $(d-2)$-faces homeomorphic to a ball or sphere. As
observed by Clara Chan~\cite{Chan} this is equivalent to stating that
there is a pair $(i,j)$ such that $F_t\cap (F_1\cap\cdots\cap F_{t-1})$ 
is the union of $i$ antipodally unpaired $(d-2)$-faces and $j$
antipodally paired $(d-2)$-faces. Here $0\leq i\leq d-1$ and if $i=0$
then $j=d-1$. (An explanation this observation may be found
in~\cite[Lemma 3.2]{Ehrenborg-Hetyei}.) We call $(i,j)$ the {\em type}
of the cubical shelling component $(i,j)$. 

The main result of~\cite{Chan}
that each cubical shelling component contributes a polynomial with
positive coefficient to the toric $h$-polynomial. Thus the toric
$h$-vector of a shellable cubical complex has non-negative
entries. Later Billera, Chan, and Liu~\cite{Billera-Chan-Liu} proved
that for {\em star-shellable cubical complexes} (these are cubical
complexes such that each shelling component has type $(i,0)$ for some
$i$) the toric $h$-vector is the $h$-vector of a shellable simplicial
complex. The proof depends on using the combinatorial description of the
toric $h$-polynomial contributions that may be found in~\cite{Chan}. 
 
\section{Explicit formulas for the toric $h$-contribution of the cubical
shelling components}
\label{s_expl}

Using the recursions given in Chan's paper~\cite{Chan} we compute an
explicit formula for the contribution of each cubical shelling component
to the toric $h$-polynomial of a shellable cubical complex. 
Our main result in this section is the following.
\begin{theorem}
\label{T_fdij}
Let $F_t$ be a facet of type $(i,j)$ in the shelling of a
$(d-1)$-dimensional cubical complex. Then adding $F_t$ to
$F_1\cup\cdots\cup F_{t-1}$ changes the polynomial $f(P,x)$ of the face
poset $P$ by
\begin{equation}
\label{E_fdij}
f_d(i,j,x):=\sum_{k=0}^{d-1} C_{d-1-k} (1-x)^k x^{d-k} \sum_{s=0}^j
\binom{j}{s}\binom{d+i+j-1-k-s}{k-s}
\end{equation}
if $j\neq d-1$, and by 
\begin{equation}
\label{E_fd0d-1}
f_d(0,d-1,x)= \sum_{k=0}^{d-1} C_{d-1-k}
\binom{d-1-k}{k} (x-1)^k\quad\mbox{if $i=0$ and $j=d-1$}.
\end{equation}
Here $C_n=\binom{2n}{n}/(n+1)$ is a Catalan number. 
\end{theorem} 
\begin{proof}
Let us verify the statement first for $(i,j)=(0,0)$, that is, for the
contribution of the first facet in the shelling. For this we want to
prove
\begin{equation}
\label{E_fd00}
f_d(0,0,x):=\sum_{k=0}^{d-1} C_{d-1-k} \binom{d-1-k}{k} (1-x)^k x^{d-k}.
\end{equation}
As noted in the proof of Lemma 1 in~\cite{Chan}, as a consequence of 
(\ref{E_fg}) we have 
$$
f_d(0,0,x)=x^d g(L_{d-1},x^{-1}).
$$
Using this observation, equation (\ref{E_fd00}) may be obtained by
rewriting (\ref{E_gd}) as 
\begin{align*}
g(L_d,x)&=
\sum_{k=0}^{\lfloor d/2\rfloor} \frac{1}{d-k+1}
\binom{d}{d-k}\binom{2d-2k}{d} (x-1)^k\\
&=\sum_{k=0}^{\lfloor d/2\rfloor}
\frac{1}{d-k+1}\binom{2d-2k}{d-k}\binom{d-k}{k},\\
\end{align*}
yielding
\begin{equation}
\label{E_gdc}
g(L_d,x)=\sum_{k=0}^{\lfloor d/2\rfloor} C_{d-k}
\binom{d-k}{k} (x-1)^k. 
\end{equation}
Next, using induction on $i$, we show that 
\begin{equation}
\label{E_fdi0}
f_d(i,0,x):=\sum_{k=0}^{d-1} C_{d-1-k} \binom{d+i-1-k}{k} (1-x)^k x^{d-k}.
\end{equation}
holds for all $i\geq 0$. For $i=0$ this is (\ref{E_fd00}) above.
According to the proof of Lemma~2 in~\cite{Chan}, the polynomials
$f_d(i,0,x)$ satisfy the recursion formula 
$$
f_d(i,0,x)=f_d(i-1,0,x)-(x-1)f_{d-1}(i-1,0,x).
$$
Thus, using the induction hypothesis for $i-1$, we obtain 
\begin{align*}
f_d(i,0,x)&= \sum_{k=0}^{d-1} C_{d-1-k} \binom{d+i-2-k}{k} (1-x)^k x^{d-k}\\
&+(1-x)\sum_{k=0}^{d-2} C_{d-2-k} \binom{d+i-3-k}{k} (1-x)^k x^{d-1-k}.\\
\end{align*}
Shifting the value $k$ by one in the second sum yields
\begin{align*}
f_d(i,0,x)&= 
\sum_{k=0}^{d-1} C_{d-1-k} \binom{d+i-2-k}{k} (1-x)^k x^{d-k}\\
&+\sum_{k=1}^{d-1} C_{d-1-k} \binom{d+i-2-k}{k-1} (1-x)^k x^{d-k}.\\
\end{align*}
Equation (\ref{E_fdi0}) is now an immediate consequence of 
$$
\binom{d+i-2-k}{k}+\binom{d+i-2-k}{k-1}=\binom{d+i-1-k}{k}.
$$
To show the statement for $i>0$ and all $j\geq 0$ we use induction on $j$. 
The basis of the induction is (\ref{E_fdi0}). 
Using the recursion formula 
$$
f_d(i,j,x)=f_d(i,j-1,x)-2(x-1) f_{d-1}(i,j-1,x),
$$ 
which may be found in the proof of Lemma~3 in~\cite{Chan}, and the
induction hypothesis for $j-1$ we obtain 
\begin{align*}
f_d(i,j,x)&= 
\sum_{k=0}^{d-1} C_{d-1-k} (1-x)^k x^{d-k} 
\sum_{s=0}^{j-1}\binom{j-1}{s}\binom{d+i+j-2-k-s}{k-s}\\
&+2(1-x)
\sum_{k=0}^{d-2} C_{d-2-k} (1-x)^k x^{d-1-k} \sum_{s=0}^{j-1}
\binom{j-1}{s}\binom{d+i+j-3-k-s}{k-s}.
\end{align*} 
Shifting the value of $k$ by one in the second sum yields
\begin{align*}
f_d(i,j,x)&= 
\sum_{k=0}^{d-1} C_{d-1-k} (1-x)^k x^{d-k} 
\sum_{s=0}^{j-1}\binom{j-1}{s}\binom{d+i+j-2-k-s}{k-s}\\
&+2
\sum_{k=1}^{d-1} C_{d-1-k} (1-x)^k x^{d-k} \sum_{s=0}^{j-1}
\binom{j-1}{s}\binom{d+i+j-2-k-s}{k-1-s}.
\end{align*} 
Since 
\begin{align*}
&\binom{d+i+j-2-k-s}{k-s}+2 \binom{d+i+j-2-k-s}{k-1-s}=\\
&\binom{d+i+j-1-k-s}{k-s}+\binom{d+i+j-2-k-s}{k-1-s},
\end{align*}
we may write
\begin{align*}
f_d(i,j,x)&= 
\sum_{k=0}^{d-1} C_{d-1-k} (1-x)^k x^{d-k} 
\sum_{s=0}^{j-1}\binom{j-1}{s} \binom{d+i+j-1-k-s}{k-s}\\
&+
\sum_{k=0}^{d-1} C_{d-1-k} (1-x)^k x^{d-k} 
\sum_{s=0}^{j-1}\binom{j-1}{s} \binom{d+i+j-2-k-s}{k-1-s}.\\
\end{align*}
Shifting the value of $s$ by one in the second sum yields
\begin{align*}
f_d(i,j,x)&= 
\sum_{k=0}^{d-1} C_{d-1-k} (1-x)^k x^{d-k} 
\sum_{s=0}^{j-1}\binom{j-1}{s} \binom{d+i+j-1-k-s}{k-s}\\
&+
\sum_{k=0}^{d-1} C_{d-1-k} (1-x)^k x^{d-k} 
\sum_{s=1}^{j}\binom{j-1}{s-1} \binom{d+i+j-1-k-s}{k-s}.\\
\end{align*}
The statement follows now from
$$
\binom{j-1}{s} +\binom{j-1}{s-1} =\binom{j}{s}. 
$$
Consider finally the case when $i=0$ and $j=d-1$. As noted
in~\cite{Chan}, we have 
$$
f_d(0,d-1,x)=g(L_{d-1},x), 
$$
thus (\ref{E_fd0d-1}) follows from (\ref{E_gdc}).
\end{proof}
Since the maximum rank that occurs in the 
face poset of a $(d-1)$-dimensional cubical complex is $d$,
Theorem~\ref{T_fdij} may be rephrased for the toric $h$-polynomial, as
defined in~\cite{Billera-Chan-Liu}, as follows.
\begin{corollary}
\label{C_fdij}
Let $F_t$ be a facet of type $(i,j)$ in the shelling of a
$(d-1)$-dimensional cubical complex. Then adding $F_t$ to
$F_1\cup\cdots\cup F_{t-1}$ changes the toric $h$-polynomial of the
cubical complex by 
\begin{equation}
\label{E_hdij}
h_d(i,j,x):=\sum_{k=0}^{d-1} C_{d-1-k} (x-1)^k \sum_{s=0}^j
\binom{j}{s}\binom{d+i+j-1-k-s}{k-s}
\end{equation}
if $j\neq d-1$, and by 
\begin{equation}
\label{E_hd0d-1}
h_d(0,d-1,x)= \sum_{k=0}^{d-1} C_{d-1-k}
\binom{d-1-k}{k} x^{d-k}(1-x)^k\quad\mbox{if $i=0$ and $j=d-1$}.
\end{equation}
\end{corollary} 
 
\section{A new plane tree enumeration model} 
\label{s_pte}

In this section we develop new plane tree model with enumeration
problems having the toric $h$-contributions of the cubical shelling
components as their answer. The model is bijectively related to Clara
Chan's~\cite{Chan} model, the bijection will be outlined in
Section~\ref{s_reflect}. Our interest in this model is due to the fact
that the formulas have relatively simple direct proofs, taking advantage
of the explicit formulas given in Section~\ref{s_expl}, and that the
definition of the third class of special vertices we use is less
technical. In our model leftmost children will play a special role, and
we will sometimes refer to them as {\em first-born} children. 

Because of ``historic reasons'' (a bijective connection to notions
introduced by Clara Chan~\cite{Chan}) and because of the case of the
shelling components of type $(0,0)$ we first introduce the following notions:  
\begin{definition}
\label{D_types}
We define three types of special vertices in a plane tree as follows.
\begin{itemize}
\item[-] A vertex $v$ is a {\em type $(0)$ special vertex}
if $v$ is a leaf, $v$ is the leftmost child of its parent, and 
the parent of $v$ is not the root. 
\item[-] A vertex $v$ is a  {\em type $(1)$ special vertex} if $v$ is a leaf 
and it is either the leftmost child of the root, or the parent of $v$ is 
not the root and $v$ is not the leftmost child of its parent.
\item[-] A vertex $v$ is a {\em type $(2)$ special vertex} if $v$ is not
  a leaf, and it is the leftmost child of its parent.
\end{itemize}
\end{definition}

The distinction between first-born leaf children of nonroot parents (type
$(0)$) and the first-born leaf child of the root (type $(1)$) seems
artificial but inevitable when one wants to state Lemma~\ref{L_GS}
below. However, to prove our enumeration statement for shelling
components of type $(i,j)$ where $i>0$, we will be more comfortable
using the following, more straightforward variant of
Definition~\ref{D_types}. 
\begin{definition}
\label{D_types2} We define three types of special vertices in a plane
tree as follows.
\begin{itemize}
\item[-] A vertex $v$ in is a {\em type $[0]$ special vertex}
if $v$ is a leaf, and $v$ is the leftmost child of its parent.
\item[-]A vertex $v$ is a  {\em type $[1]$ special vertex} if $v$ is a
  leaf and it is not the leftmost child of its parent. 
\item[-] A vertex $v$ is a {\em type $[2]$ special vertex} if $v$ is not
  a leaf, and it is the leftmost child of its parent.
\end{itemize}
\end{definition}
Both definitions define disjoint sets of special vertices, and the
definition of type $(2)$ is the same as the definition of type
$[2]$. For those trees where the leftmost child of the root is not a
leaf, the two definitions coincide. 

We have the following variant of the Gessel-Shapiro result. 
\begin{lemma}
\label{L_GS}
The coefficient of $x^k$ in $g(L_d,x)$ is the number of plane trees on
$d+1$ vertices with exactly $k$ type $(0)$ special vertices. 
\end{lemma} 
\begin{proof}
The statement may be rephrased as follows:
  $g(L_d,x)$ is the total weight of all plane trees on $d+1$ vertices,
  where the weight of each tree is the product of the weight of its
  vertices, each type $(0)$ special vertex  has weight $x$, all other
vertices have weight one. 

Replacing $x$ with $x+1$ yields the following
equivalent statement: the coefficient $x^k$ in
$g(L_d,x+1)$ is the number of plane trees on
$d+1$ vertices with exactly $k$ {\em marked} type $(0)$ special vertices. 
In fact, replacing each $x$ with $x+1$ corresponds to choosing whether 
to mark or not to mark each type $(0)$ special vertex.

Substituting $x+1$ into $x$ in (\ref{E_gdc}) we obtain we formula 
$$
g(L_d,x+1)=\sum_{k=0}^{\lfloor d/2\rfloor} C_{d-k}
\binom{d-k}{k} x^k. 
$$
Using this formula, the last equivalent rephrasing of the lemma is
easy to show. In fact, every plane tree $T$ with $k$ marked type $(0)$
vertices  
may be uniquely given by determining the plane tree $T'$ on $d+1-k$ vertices
obtained by deleting the marked leaves, and the $k$-element subset $S$
of nonroot vertices of $T'$ that are parents of a marked vertex. 
Note that no two marked vertices have the same parent, since each is the
leftmost child of its parent. Given a plane tree $T'$ on $d+1-k$
vertices we may select any $k$-element subset of its nonroot vertices as 
the parents which obtain a new leftmost and marked child. Therefore
there are $C_{d-k} \binom{d-k}{k}$ plane trees with $k$ marked type $(0)$
special vertices.  
\end{proof}
Using Theorem~\ref{T_fdij} we obtain the following consequences of 
Lemma~\ref{L_GS}.
\begin{corollary}
\label{C_00}
The coefficient of $x^{d-k}$ in $f_d(0,0,x)$ is 
the number of plane trees on $d$ vertices with $k$ type $(0)$ special
vertices.
\end{corollary}
\begin{corollary}
The coefficient of $x^k$ in $f_d(0,d-1,x)$ is 
the number of plane trees on $(d+1)$ vertices with $k$ type $(0)$ special
vertices.
\end{corollary}
Just like in the work of Chan~\cite{Chan} 
we describe the coefficient of $x^{d-k}$ in $f_d(i,j,x)$ for other
values of $i$ and $j$ by introducing a {\em labeling} of the vertices.
In~\cite{Chan} the {\em preorder} labeling was used for this purpose. We
will use the {\em postorder} labeling. This may be defined recursively
as follows. First remove the root, and label the resulting trees in
postorder, left to right. Then we label the root with the highest
number. An example is shown in Figure~\ref{F_postorder}. The meaning of
the bold edges will be explained in Section~\ref{s_nc}.
\begin{figure}[h]
\begin{center}
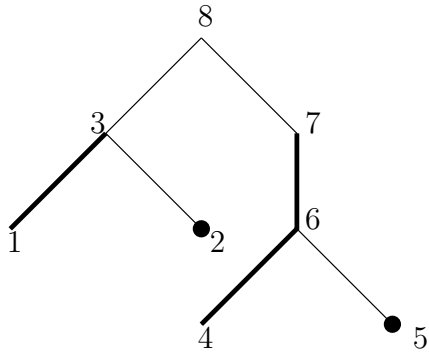
\end{center}
\caption{Postorder tree}
\label{F_postorder}
\end{figure}

For each positive integer $n$ let us fix an algebraic operation $F_n$ 
with $n$ variables. Let us use ``Polish notation'', that is, write 
$a_1 a_2\cdots a_n F_n$ instead of $F_n(a_1,\ldots, a_n)$. There is an
obvious bijection between plane trees on $d$ vertices and polynomial
expressions of length $d$, using a single variable $x$ and the operations 
$F_1,F_2,\ldots$. In fact, we can think of leaves as nodes representing 
the variable $x$, and of each non-leaf with $n$ children as an operation
$F_n$. Each operation represented by a non-leaf takes the outputs
calculated at its children (listed left to right) as its input. Under
this bijection, the postorder tree in Figure~\ref{F_postorder}
corresponds to the polynomial $xxF_2xxF_2F_1F_2$. Each letter $u$,
except for the last one, is substituted into a variable of the operation
associated to its parent node, which we call its {\em parent
  operation}. The root in the tree corresponds to the last letter in the
expression, we will thus call the operation 
represented by the last letter the {\em root operation}. In the rest of
the paper we will call the polynomial expression assigned to a given
plane tree $T$ by the bijection described above the {\em associated
  polynomial expression} and denote it by $P(T)$. 

In $P(T)$, type $(0)$ special vertices
correspond to all letters $x$ that are substituted into the first
variable of their  parent operation, provided the parent operation is
not the root operation. Type $(1)$ special vertices correspond to 
all letters $x$ that are either substituted into the first
variable of the root operation, or into a non-first variable of a
nonroot operation. Type $(2)$ special vertices correspond to operation
symbols which are substituted into the first variable of their parent
operation. The correspondence for type $[0]$ and type $[1]$ vertices is
similar but easier. 

Consider plane trees as Hasse diagrams of partial orders with each
parent being greater than its child. Removing a non-maximum element from
a partial order corresponds to removing a nonroot vertex $v$ in the
plane tree, and turning the children of $v$ (if any) into children of
the parent of $v$. Let us define the {\em removal of $v$} in such a way
that it also preserves the plane order: when removing $v$, replace $v$
in the ordered list of the children of the parent of $v$ with the
ordered list of the children of $v$. This removal operation results in a
smaller plane tree.  

The effect of the 
removal of a single vertex from a plane tree $T$ on $P(T)$ may
be very easily described. In fact, if the removed 
vertex has $c\geq 0$ children and $s\geq 0$ siblings then in $P(T)$ 
we need to remove the letter associated to the
removed vertex and replace its parent operation $F_{s+1}$ with $F_{s+c}$.

\begin{proposition}
\label{P_restr}
Let $T$ be a plane tree and $\{v_1,\ldots,v_k\}$ be a subset of its
nonroot vertices. Then removing these vertices in any order 
$v_{\pi(1)},\ldots, v_{\pi(k)}$ results in the same plane tree. Here 
$\pi$ is any permutation of $\{1,2,\ldots,k\}$.
\end{proposition}
\begin{proof}
The removal operation does not change the induced partial order
on the remaining vertices. Thus the resulting tree will always induce 
the same partial order on the remaining vertices. The plane order does not
change either since in the associated polynomial
expressions the order of the variables associated to the remaining
vertices does not change under the removal operation. 
\end{proof}

\begin{definition}
We say that the plane tree $T$ is obtained from 
a plane tree $T'$ by {\em inserting} a vertex $v$ if $v$ is a nonroot
vertex of $T$ and removing $v$ from $T$ yields $T'$.  
\end{definition}

Whereas the removal of a vertex is uniquely defined, there may be more
than one way to insert a nonroot vertex $v$ into a plane tree such that 
we obtain a plane tree. The following two {\em insertion lemmas}
describe two situations where the insertion at a prescribed postorder 
position may be done in one and only one way, if we want to create a
special vertex of a certain type. The first lemma states that there is a
unique way to insert a type $(1)$ special vertex at a given position,
and this will not change the special properties of the vertices that
precede in postorder. 

\begin{lemma}[First insertion lemma]
\label{L_t1}
Let $T'$ be a plane tree on $d$ vertices and $p$ an integer satisfying
$1\leq p\leq d$. Then there is a unique plane tree $T$ on $d+1$ vertices
such that the vertex $u$ whose postorder label is $p$ in $T$ is a type $(1)$
special vertex and removing this vertex from $T$ yields
$T'$. Furthermore, for any $\tau\in\{0,1,2\}$ a  nonroot vertex $v\neq
u$ whose postorder label is at most $p-1$ (in either $T'$ or $T$)
is a type $(\tau)$ special vertex in $T'$ if and only if it has the same
property in $T$.  
\end{lemma}
\begin{proof}
As noted before Proposition~\ref{P_restr}, after removing the $p$-th
letter from $P(T)$, the resulting 
word must differ at at most one position from $P(T')$. The letter that
is removed from $P(T)$ must be a variable $x$. 
This letter corresponds to a type $(1)$ special vertex if and only
if either $p=1$ and the letter is the first input of the root operation
or $p>1$ and the letter at the $(p-1)$-th place belongs to the same parent
operation. Either way there is exactly one way to change a $P(T')$ to obtain 
a $P(T)$ such that $T$ and $T'$ have the required relation: we must
insert a new $p$-th letter $x$ and, for $p=1$ we must increase the
number of variables of the root operation by one,  whereas for $p>1$ we
must increase the number of variables of the parent of the letter at the
$(p-1)$-th position by one. 

For example if $p=1$ and $P(T')=xxx F_3 F_1$ then the only way to insert
a new first leaf as a type $(1)$ special vertex is to create the plane
tree defined by $P(T)=xxxxF_3F_2$. If $p=4$ and $P(T')=xxF_2xxF_3xF_2$
then we must have $P(T)=xxF_2xxxF_4xF_2$.

To prove the last part of the Lemma, observe first 
that the vertices $v\neq u$ having the same postorder label in $T'$ and in
$T$ are exactly the vertices whose postorder label is at most $p-1$.
We are inserting  
$u$ as a leaf, and the parent of $u$ is either the root, or a vertex
that already has at least one child (the $(p-1)$-th vertex in
postorder). Thus a nonroot vertex $v\neq u$ is a leaf in $T'$ if and
only if it is a leaf in $T$. Similarly $v\neq u$ is a root child in $T'$
if and only if it is a root child in $T$. 
If $p\neq 1$ then $u$ is not inserted as the first
child, thus any $v\neq u$ is the first child of its parent in $T'$ if
and only if the same holds in $T$. Therefore in the case when $p\neq 1$,
we have the stronger statement that any nonroot vertex $v\neq
u$ is a type $(\tau)$ special vertex in $T'$ if and only if it has the same
property in $T$. Finally if $p=1$ then the last part of the Lemma is
trivially true, since no nonroot vertex $v\neq u$ has postorder label
less than $1$. 
\end{proof}
Lemma~\ref{L_t1} is a strong statement, and it may be rephrased
as follows. 
\begin{corollary}
\label{C_t1}
Let $T'$ be a plane tree on $d$ vertices and $p$ an integer satisfying
$2\leq p\leq d$. Then there is a unique plane tree $T$ on $d+1$ vertices
such that the vertex $u$ whose postorder label is $p$ in $T$ is a type $[1]$
special vertex and removing this vertex from $T$ yields
$T'$. Furthermore, for any $\tau\in\{0,1,2\}$ a  nonroot vertex $v\neq
u$ whose postorder label is at most $p-1$ (in either $T'$ or $T$)
is a type $[\tau]$ special vertex in $T'$ if and only if it has the same
property in $T$.  
\end{corollary}
In fact, for $p\geq 2$ the vertex we are about to insert can not be 
the leftmost leaf whose parent is the root. Thus the unique possibility
of inserting a type $(1)$ special vertex is equivalent to the unique
possibility  of inserting a type special $[1]$ vertex. For the preceding
vertices, a vertex is a type $(\tau)$ special vertex if and only if 
it is a type $[\tau]$ special vertex, except for the first vertex 
in postorder, if it is a leaf. Being such a vertex is preserved by the
insertion or removal of $u$. 

The second insertion lemma states that there is a unique way to insert 
a type $(2)$  (that is, type $[2]$) special vertex at a prescribed
position $p\geq 2$. Alas it is not possible to state that this insertion
has to leave the special types $(\tau)$ where $\tau\in \{0,1,2\}$
unchanged for the preceding vertices, because of the following
example. Assume we are given a tree $T'$ such that  the first vertex
$v_1$ in postorder is a leaf and the child of the root.  As we will see
in the proof of Lemma~\ref{L_t2}, the only way to insert 
a type $(2)$ special vertex as the second vertex, is by making this 
new vertex a root child and a parent of $v_1$. In the new tree, $v_1$ is
not a type $(1)$ special vertex any more but a type $(0)$ special
vertex. However, this complication may be eliminated by
making the statement about type $[\tau]$ vertices. 

\begin{lemma}[Second insertion lemma]
\label{L_t2}
Let $T'$ be a plane tree on $d$ vertices and $p$ an integer satisfying
$2\leq p\leq d$. Then there is a unique plane tree $T$ on $d+1$ vertices
such that the vertex whose postorder label is $p$ in $T$ is a type $[2]$
special vertex and removing this vertex from $T$ yields $T'$. 
Furthermore, for any $\tau\in\{0,1,2\}$ if nonroot vertex $v\neq u$ 
whose postorder label is at most $p-1$ (in either $T'$ or $T$)
is a type $[\tau]$ special vertex in $T'$ if and only if it has the same
property in $T$.  
\end{lemma}
\begin{proof}
In analogy to the proof of Lemma~\ref{L_t1}, the first part of the 
statement is easily
shown by considering $P(T)$ and $P(T')$. Given a polynomial expression
$P(T')$, now we have to insert an operation $F_c$ as the new $p$-th
letter, and the $(p-1)$-th letter $L$ must have the inserted operation as
its parent in $P(T)$. The index $c$ of $F_c$ is determined by the fact
that $L$ must be the $c$-th child of its parent $L'$ in $P(T')$. Finally,
inserting $F_c$ at the $p$-th position decreases the number of variables 
of $L'$ by $c-1$. For example, if $p=4$ and $P(T')=xF_1xxF_3$ then the 
third letter $x$ is the second child of its parent operation $F_3$. The
only way to get $P(T)$ is to insert an $F_2$ as a new fourth letter and 
to change the last $F_3$ to $F_2$. Thus we must have $P(T)=xF_1xF_2xF_2$.

To prove the last part of the Lemma, observe that the vertices $v\neq u$ whose
postorder label is at most $p-1$ do not lose any of their children nor 
do they lose any preceding sibling at the insertion of $u$.  
\end{proof}

Repeated use of the two insertion lemmas yields the following.
\begin{proposition}
\label{P_insert}
Let $T'$ be a plane tree on $d+1-k$ vertices,
$\{p_1,\ldots,p_k\}\subseteq \{1,2,\ldots,d\}$ and $(\tau_1,\ldots,\tau_k)\in
\{1,2\}^k$. Then there is a unique plane tree $T$ on $d+1$ vertices such
that for each $s\in\{1,2,\ldots,k\}$ the vertex $v_s$ of $T$ whose postorder
label is $p_s$ is a type $[\tau_s]$ special vertex, and removing the
vertices $v_1,\ldots,v_k$ from $T$ yields $T'$. 
\end{proposition}
\begin{proof}
Note that by Proposition~\ref{P_restr} the order of removing the
vertices $v_1,\ldots, v_k$ is irrelevant. Without loss of generality we
may assume $p_1<\cdots <p_k$. We proceed by induction on $k$. For
$k=1$ and $\tau_1=1$ the statement is identical with Corollary~\ref{C_t1},
whereas setting $k=1$ and $\tau_1=2$ makes the statement identical with
Lemma~\ref{L_t2}. Assume the statement is true for $k-1$. Let $T'$
be a plane tree on $d+1-k$ vertices,   $\{p_1,\ldots,p_k\}\subseteq
\{1,2,\ldots,d\}$ and $(\tau_1,\ldots,\tau_k)\in\{1,2\}^k$. If there is 
any tree $T$ satisfying the stated conditions, the tree $T''$ obtained by 
removing $v_k$ from $T$ must be a tree on $d$ vertices 
such that for each $s\in\{1,2,\ldots,k-1\}$ the vertex $v_s$ of $T''$ whose
postorder label is $p_s$ is a type $[\tau_s]$ special vertex, and removing the
vertices $v_1,\ldots,v_{k-1}$ from $T''$ yields $T'$. By our induction
hypothesis, $T''$ may be obtained from $T'$ by inserting the vertices 
$v_1,\ldots,v_{k-1}$  appropriately. Given $T''$ there is a
unique way to insert $v_k$ as a type $(\tau_k)$ special vertex 
because of the already shown $k=1$ case of this Proposition. Finally
note that by Corollary~\ref{C_t1} and Lemma~\ref{L_t2}, inserting $v_k$
can not change the property of being type $[\tau_s]$ special of any
preceding vertex $v_s$.  
\end{proof}

The main result of this section is the following.
\begin{theorem}
\label{T_main1}
Let $(i,j)$ be the type of a shelling component in a $(d-1)$-dimensional
cubical complex and assume $j<d-1$. Then 
the coefficient of $x^{d-m}$ in $f_d(i,j,x)$ is the number of plane
trees on $d$ vertices with exactly $m$ vertices $v$ having (exactly)
one of the following properties:
\begin{itemize}
\item[--] $v$ is a type $(0)$ special vertex;
\item[--] $v$ is a type $(1)$ special vertex whose label in postorder 
is at most $i$ or at least $d-j$;
\item[--] $v$ is a type $(2)$ special vertex whose label in postorder
  is at least $d-j$.  
\end{itemize}
\end{theorem}
For $i=0$, requiring $j<d-1$ forces $(i,j)=(0,0)$, and 
Theorem~\ref{T_main1} becomes Corollary~\ref{C_00}. For $i>0$,
Theorem~\ref{T_main1} becomes equivalent to the following statement,
which seems easier to prove. 
\begin{theorem}
\label{T_main2}
Let $(i,j)$ be the type of a shelling component in a $(d-1)$-dimensional
cubical complex and assume $j<d-1$. Then 
the coefficient of $x^{d-m}$ in $f_d(i,j,x)$ is the number of plane
trees on $d$ vertices with exactly $m$ vertices $v$ having (exactly)
one of the following properties:
\begin{itemize}
\item[--] $v$ is a type $[0]$ special vertex;
\item[--] $v$ is a type $[1]$ special vertex whose label in postorder 
is at most $i$ or at least $d-j$;
\item[--] $v$ is a type $[2]$ special vertex whose label in postorder
  is at least $d-j$.  
\end{itemize}
\end{theorem}
\begin{proof}
By equation (\ref{E_fdij}) we have
\begin{equation}
\label{E_fdijc}
[x^{d-m}]f_d(i,j,x)=\sum_{k=m}^{d-1} (-1)^{k-m} C_{d-1-k}
\binom{k}{k-m}\sum_{s=0}^j \binom{j}{s}\binom{d+i+j-1-k-s}{k-s}.
\end{equation}
Thus it is sufficient to show that the number of plane trees on $d$
vertices with $k$ {\em marked} special vertices having the properties
listed in the Theorem is 
$$
C_{d-1-k}\sum_{s=0}^j \binom{j}{s}\binom{d+i+j-1-k-s}{k-s}.
$$
The statement will then follow from (\ref{E_fdijc}) by
inclusion-exclusion. 

Consider a plane tree $T$ with with $k$
marked special vertices. Let us remove the marked vertices (the order of
removing them does not matter by Proposition~\ref{P_restr}). We will be
left with a plane tree $T'$ on $d-k$ vertices, which is one of 
$C_{d-1-k}$ possible trees. Thus it suffices to
show that for each plane tree $T'$ on $d-k$ vertices there are 
$$
\sum_{s=0}^j \binom{j}{s}\binom{d+i+j-1-k-s}{k-s}
$$
plane trees $T$ with $k$ marked special vertices that yield $T'$ after
removing the marked vertices. Assume exactly $k_{\tau}$ of the marked
vertices have type $[\tau]$ for $t=0,1,2$. (Thus $k=k_0+k_1+k_2$.) 
It is sufficient to show that for any fixed triplet $(k_0,k_1,k_2)$ of
natural numbers satisfying $k_0+k_1+k_2=k$, there
are 
$$
\binom{d-k}{k_0}\binom{j}{k_2}\binom{i-1+j-k_2}{k_1}
$$
plane trees $T$ with $k_{\tau}$ marked type $[\tau]$ special vertices with the 
stated properties such that
removing the marked special vertices results in $T'$. The statement 
then follows from
$$
\sum_{k_0+k_1+k_2=k} \binom{d-k}{k_0}\binom{j}{k_2}\binom{i-1+j-k_2}{k_1}
= \sum_{s=0}^j \binom{j}{s}\binom{d+i+j-1-k-s}{k-s}.
$$
Let us reinsert first the marked type $[0]$ special vertices into $T'$ 
and then the other marked special vertices. (We are allowed to fix such
an order by Proposition~\ref{P_restr}.) In analogy to the proof of
Lemma~\ref{L_GS}, there are exactly $\binom{d-k}{k_0}$ ways to insert 
$k_0$ marked type $[0]$ vertices into $T'$, to obtain a plane tree $T''$
with $d-k+k_0$ vertices containing $k_0$ marked type $[0]$ special vertices
such that the removal of the marked vertices from $T''$ results in
$T'$. Thus we are left to show that there are exactly
$$\binom{j}{k_2}\binom{i+j-k_2}{k_1}$$ 
ways to insert marked type $[1]$ and type
$[2]$ special vertices to obtain a plane tree $T$ satisfying our
requirements. By the stated restrictions, there are $j$ positions
($d-j,\ldots,d-1$) where $T$ may contain a marked type $[2]$ special
vertex. Once we selected these positions, there are $i-1+j-k_2$ positions
($2,\ldots, i$ and the ones not yet selected from 
$\{d-j,\ldots, d-1\}$) where $T$ may contain a marked type
$[1]$ special vertex. Once we select the positions where we want to
have our marked type $[1]$ and type $[2]$ vertices, by
Proposition~\ref{P_insert} there is a unique way to insert the marked
special vertices at the prescribed positions, no matter how we selected 
these positions. 
\end{proof}

\section{Connection between the two plane tree enumeration models}
\label{s_reflect} 

In this section we outline a bijection between our model and Chan's
model~\cite{Chan}, showing that the proofs of the present paper could be
directly translated into proofs of Chan's results and vice versa. 

We begin with a preorder plane tree from Chan's paper~\cite{Chan}, shown
in Figure~\ref{F_ctree}. 
\begin{figure}[h]
\begin{center}
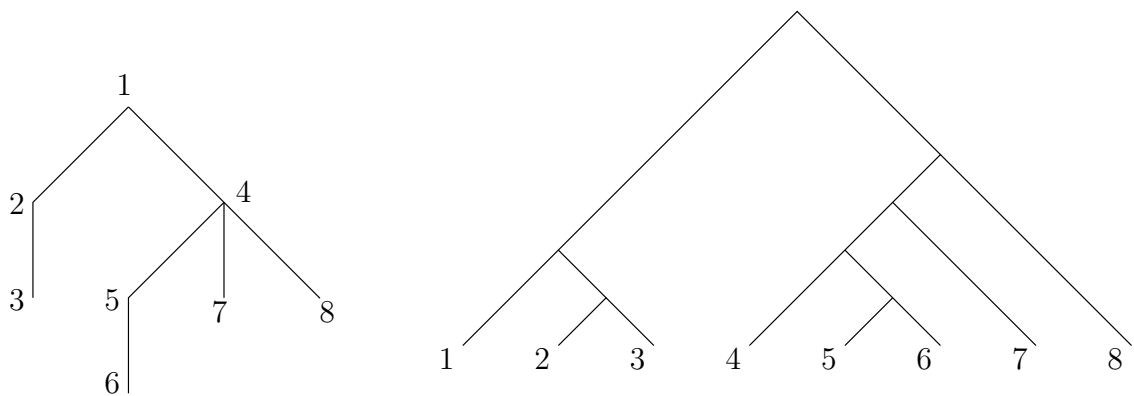
\end{center}
\caption{Chan's example and the associated binary tree}
\label{F_ctree}
\end{figure}
It is well known that plane trees on $d$ vertices are in bijection with
Catalan paths of length $2d$, and a bijection may be given by walking
around the tree in counterclockwise order, keeping very close to its
edges, and recording an ``up'' (or $+$) for each step when along the
nearest edge our move represents moving away from the root, and a
``down'' (or $-$) otherwise. Thus the plane tree shown in
Figure~\ref{F_ctree} corresponds to the sequence $++--+++--+-+--$. Let
us think now of each $+$ as ``opening a parenthesis'' and of each $-$ as
``closing a parenthesis''. We obtain a parenthesization of the product 
$x_1\cdots x_d$ which may be represented by a binary tree as shown in
the right hand side of Figure~\ref{F_ctree}. Let us reflect the
binary tree about a vertical axis, we then obtain a binary tree 
with decreasing labels, such as the one in 
Figure~\ref{F_mirror}.  Using the bijection described above in the
opposite direction, we obtain the mirror image of a postorder tree. This
postorder tree is the planar mirror image of the one shown in
Figure~\ref{F_postorder}. 

\begin{figure}[h]
\begin{center}
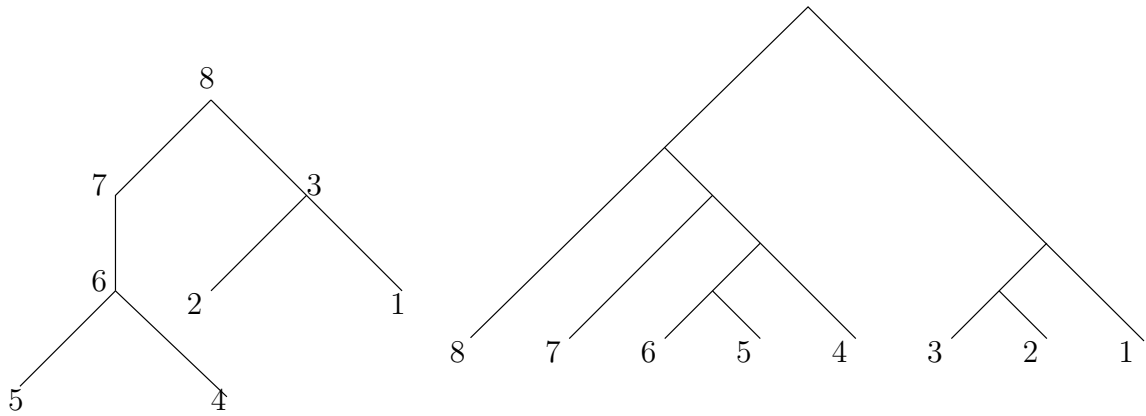
\end{center}
\caption{Mirrored postorder tree associated to the mirrored binary tree}
\label{F_mirror}
\end{figure}

\begin{proposition}
The sequence of transformations described above establishes a bijection
between preorder trees and postorder trees. 
\end{proposition}
\begin{proof} 
When walking around a preorder tree, the steps marked $+$ represent 
visiting a vertex that was nor visited before, whereas the steps marked
$-$ represent ``backtracking''. When we reflect the binary tree in the
plane, the associated parenthesization, along with the associated sequence of  
$+$ and $-$ steps go into their mirror image. When walking around 
a mirrored postorder tree, the steps marked $+$ represent again 
visiting a vertex that was nor visited before, whereas the steps marked
$-$ represent ``backtracking''. The only difference is that this time we
visit the highest labeled vertex first and we number the vertices in
decreasing order along the way. It is not hard to complete the proof
using these observations, the details are left to the reader. 
\end{proof}

\begin{proposition}
The bijection described above maps the set of forks in a preorder tree
onto the set of type $(0)$ special vertices in the corresponding
postorder tree.  
\end{proposition}
\begin{proof}
A vertex $v$ is a fork in the preorder tree if and only if $v$ 
and its parent are both left children in the associated binary tree. 
In fact, a vertex is a leaf if and only if it corresponds
to a right child in the binary tree, thus $v$ corresponding to a left
child is equivalent to saying that $v$ has children. If we follow the
path from $v$ to its parent, grandparent and so on in the binary tree, 
the first time we reach a right child is the time when we have to
``backtrack'' in the preorder tree. Up until that point the descendants 
of the visited ancestors in the binary tree are also the descendants of
$v$, and the leftmost descendants of the children of the visited
vertices in the binary tree are the children of $v$ in the preorder
tree. 

Similarly, $v$ is a leaf and a rightmost child of a nonroot vertex in
a mirrored postorder tree if and only if $v$ is a right child of its
parent, and the parent is a right child of its parent in the
corresponding binary tree. In fact, as noted above, $v$ being a right
child in the binary tree is equivalent to $v$ being a leaf in the
corresponding plane tree, whereas the parent of $v$ being a right child
is equivalent to saying that $v$ is the last child of its parent, and
the parent of $v$ is not the root.
\end{proof}

\begin{remark}
{\em 
As a consequence we obtain that the Gessel-Shapiro
result~\cite[Ex. 3.71g]{Stanley-EC1}, shown by Chan~\cite{Chan} by
solving a quadratic equation for a generating function is a ``mirror
image'' of Lemma~\ref{L_GS}, which was shown essentially by
inclusion-exclusion. The difference lays perhaps in the fact, that it is
easy to visualize how to ``attach a first-born marked leaf'' to any
vertex, whereas ``inserting a marked fork'' seems to be less
intuitive. However, the diligent reader should be able to transform  
the leaf insertion process into a fork insertion process, by tracing back
along the reverse of the transformation described in this section.
}\end{remark}

We conclude this section by noting that the special
vertices marked $1',\ldots, i'$ in Chan's work~\cite{Chan} correspond
exactly to our type $(1)$ special vertices, whereas the special vertices 
marked $1'', 2'', \ldots, j''$ correspond exactly to our type $(2)$
special vertices. The proof of these facts is left to the reader as an
exercise. 

\section{Noncrossing partitions and a surprising coincidence}
\label{s_nc}

A partition of $\{1,2,\ldots,d\}$ is {\em noncrossing} if for any four
elements $a<b<c<d$ the following condition is satisfied: if $a,c$ are in
the same class and $b,d$ are in the same 
class then $a,b,c,d$ are in the same class. The number of noncrossing
partitions of $\{1,2,\ldots,d\}$ is the Catalan number $C_d$, the same
as the number of plane trees on $d+1$ vertices. An explicit bijection
may be given as follows. 

\begin{definition}
Let $u$ and $v$ be nonroot vertices in a plane tree. We say that $v$ is
a {\em favorite ancestor} of $u$ if $v$ is an ancestor of $u$ and the
unique path from $v$ to $u$ involves at each step selecting the leftmost
child. Under these circumstances, we also call $u$ a {\em favorite
  descendant} of $v$.  
\end{definition} 

The choice of the terminology is motivated by the fact that in many
societies special importance is attached to the first-born child, and
many parents feel special affection towards their youngest. Such
choices correspond to always favoring the leftmost (or rightmost) child
in the family tree. Being a favorite ancestor-favorite descendant pair is a
transitive relation, its symmetrization is an equivalence relation. 
Let us call this equivalence being {\em ``favorite relatives''}. 
The equivalence classes are singletons or paths in the tree. The bold edges
and vertices marked in bold in Figure~\ref{F_postorder} represent the
equivalence classes of this equivalence relation. After labeling the
vertices in postorder, the equivalence classes become the noncrossing
partition $ 1 3/ 2/4 6 7/ 5$.

\begin{proposition}
\label{P_nc}
The ``favorite relative'' equivalence classes of nonroot vertices of a
postorder tree  on $\{1,2,\ldots, n+1\}$ form a noncrossing
partition. Associating to each  postorder tree the noncrossing
partition of its ``favorite relative'' equivalence classes is a
bijection between plane trees on $\{1,\ldots, n+1\}$ and noncrossing
partitions of $\{1,2,\ldots, n\}$. 
\end{proposition} 
\begin{proof}
Let us verify first that we obtain a noncrossing partition. Assume 
the nonroot vertices $a,b,c,d$ satisfy $a<b<c<d$ in postorder, $a$ and
$c$ are favorite relatives and $b$ and $d$ are favorite relatives. 
By the nature of the postorder labeling, $c$ is then an ancestor of $a$,
and it must also be an ancestor of $b$ since all vertices preceding $c$
in postorder that are not descendants of $c$ precede all descendants of
$c$ as well, and $b$ does not precede $a$. This means that $c$ belongs 
to the unique path from $d$ to $b$. This path involves only
choosing the leftmost child in each step, so $b$ is a favorite descendant
of $c$, and $a$, $b$, and $c$ belong to the same equivalence class. 
This class also contains $d$ since $b$ is equivalent to $d$.

For the converse it is sufficient to show that each  noncrossing
partition of $\{1,2,\ldots, n\}$ arises as the collection of
equivalence classes of favorite relatives thus the map is onto. We know
that the two sets (plane trees on $\{1,2,\ldots, n+1\}$ and non-crossing 
partitions of $\{1,2,\ldots, n\}$) have the same cardinality, so an onto
map between them is a bijection. Consider thus a noncrossing partition 
$\pi$ on  $\{1,2,\ldots, n\}$. Represent the vertices $1$, $2$, \ldots, $n+1$
as numbers on the number line. For each block
$\{a_1,\ldots,a_k\}\in\pi$, where $a_1<\cdots <a_k$, create a path 
$a_1- a_2 -\cdots -a_k$. Represent the edges $a_i-a_{i+1}$ as upper
semicircles, as they are shown with continuous lines in Figure~\ref{F_pnc}.
For each $i<k$ define the parent of $a_i$ as to be $a_{i+1}$ 
The fact that $\pi$ is a noncrossing partition is equivalent to saying 
that no two semicircles introduced up to this point intersect in an
interior point. If $a$ is the maximum element of a block, 
define its parent as the smallest $b>a$ such that the block containing $b$  
contains a $c<a$. If no such $b$ exists then select the root to be 
the parent of $a$. Represent each edge $a-b$ also with an upper semicircle
as the ones shown with dashed lines in Figure~\ref{F_pnc}. 
\begin{figure}[h]
\begin{center}
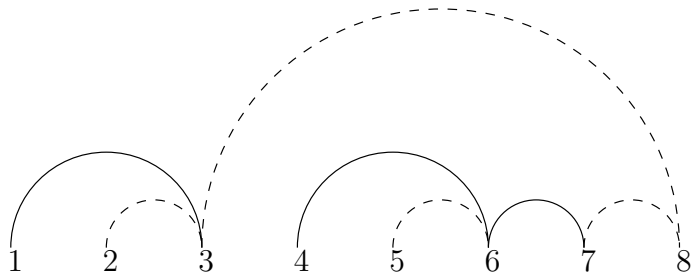
\end{center}
\caption{Postorder tree generated by a noncrossing partition}
\label{F_pnc}
\end{figure}
Even after adding the semicircles with dashed lines, the arcs will not
cross at an interior point. In fact, assume by way of contradiction that 
adding an arc $a-b$ where $a$ is the maximum element of a block and $b$
is its parent, as defined above, creates a crossing at an interior
point. If the arc $a-b$ intersects an earlier $a'-b'$ in such a way,
then these satisfy $a'<a<b'<b$. By the definition of the selection
of $b'$,  there  is a $c'<a'$ such that $c'$ and $b'$ belong to
the same block. But then $b$ is not the smallest element above $a$ whose
block contains an element preceding $a$. This contradiction shows that
adding a new dashed arc never creates an intersection of dashed arcs 
at an interior point. We are left only with the possibility that 
the arc $a-b$ intersects an arc $a'_i-a'_{i+1}$ where $a'_i$ and $a'_{i+1}$
are subsequent elements in a block of $\pi$. Then we have either 
$a<a'_i<b<a'_{i+1}$ or $a'_i<a<a'_{i+1}<b$. In the first case 
$c<a'_i<b<a'_{i+1}$ implies that $c,a'_i,a'_{i+1}$ and $b$ all belong to
the same block of $\pi$, and we should have chosen $a'_i$ or an even
smaller number to be the parent of $a$. In the second case, by 
$a'_i<a<a'_{i+1}$, the number $a'_{i+1}$ or an even smaller number should
have been chosen as the parent of $a$, instead of $b$. We obtained a
contradiction in all cases.

We obtain a rooted tree, since each vertex has a uniquely defined
parent, and every edge connects some vertex to its parent. 
For each vertex, the parent is larger than the vertex itself,.
Choosing the above representation we obtain a plane tree, since the arcs do not
cross at an interior point. It is easy to show that each bold edge 
connects the parent to its ``leftmost'' (in the picture: ``uppermost'')
child. Moreover, an arc to a smaller child passes ``to the left'' (in
the picture: ``above'') of an arc to a larger child of the same parent. 
Thus the tree has
postorder labeling, and 
for any vertex $v$ the smallest labeled child of $v$ is the one 
that is in the same equivalence class. (The minimum elements of the
equivalence classes are leaves.) Thus the collection of  equivalence
classes of favorite relatives  is $\pi$.
\end{proof}
\begin{remark}
{\em Using the bijection between preorder and postorder trees
  introduced in Section~\ref{s_reflect} we may transform 
  Proposition~\ref{P_nc} into the following result. Consider two
  non-root vertices in a preorder tree equivalent if they have the same 
parent. The equivalence classes associated to any preorder tree form a
  noncrossing partition and associating to each preorder tree the
  noncrossing partition of these equivalence classes defines a
  bijection between plane trees and noncrossing partitions. This result
  is mentioned in~\cite[Ex. 5.35a]{Stanley-EC2}. The first proof on
  record is due to Kreweras~\cite{Kreweras}. Under this correspondence,
  the nonsingleton blocks of a noncrossing partition correspond to the
  set of all children of a fork in the preorder tree. These blocks play
  the essential role in the work of Billera, Chan, and
  Liu~\cite{Billera-Chan-Liu} when they construct a simplicial complex
  whose $h$-vector is the toric $h$-vector of a cube. 
}
\end{remark}

It is easy to verify that under the correspondence introduced above,
type $(0)$ special vertices correspond exactly to the minimum elements
of the nonsingleton blocks in the corresponding noncrossing
partition. Thus we obtain the following variant 
of the Gessel-Shapiro result and Lemma~\ref{L_GS}.

\begin{lemma}
\label{L_GSnc}
The coefficient of $x^k$ in $g(L_d,x)$ is the number of noncrossing
partitions on $\{1,\ldots,d\}$ with exactly $k$ nonsingleton blocks. 
\end{lemma}
Similarly, type $(1)$ special vertices correspond to singleton blocks. In
fact, $a$ forms a singleton block exactly when $a$ has no descendants at
all (thus it is a leaf) and it is not the ``favorite child'' of its
parent, or its parent is the root. 
Type $(2)$ special vertices correspond to elements that are not the 
minimal element in their block, and not the maximal either unless the
block contains $1$. In fact, being in a nonsingleton
block and not being the minimum element of the block is equivalent to
having descendants, whereas not being the maximum element in the block 
is equivalent to being the favorite child of a nonroot parent.
Finally, the leftmost child of the root is the maximum element of 
the block containing $1$.  
Therefore we obtain the following rephrasing of Theorem~\ref{T_main1}
\begin{theorem}
\label{T_mainnc}
Let $(i,j)$ be the type of a shelling component in a $(d-1)$-dimensional
cubical complex and assume $j<d-1$. Then 
the coefficient of $x^{d-m}$ in $f_d(i,j,x)$ is the number of
noncrossing partitions on $\{1,2,\ldots,d-1\}$
with exactly $m$ elements $v$ having (exactly)
one of the following properties:
\begin{itemize}
\item[-] $v$ is the minimum element of a nonsingleton block;
\item[-] $v$ is at most $i$ or at least $d-j$ and $\{v\}$ is a
  singleton block;
\item[-] $v$ is at least $d-j$, $v$ is not the minimum element of its
  block, and if $v$ is the maximum element of its block then the block
  contains $1$.
\end{itemize}
\end{theorem}

The enumeration problems associated to the results in this section are
reminiscent of the problems discussed in the work of Yano and
Yoshida~\cite{Yano-Yoshida}, and Denise and
Simion~\cite{Denise-Simion}. Perhaps it is thus not surprising that the
toric $g$-polynomial of the $d$-cube explicitly appears in the work 
of Denise and Simion~\cite{Denise-Simion}! What is perhaps surprising
that this ``coincidence'' does not seem to have been noted earlier. 
In fact, Corollary 3.8 in~\cite{Denise-Simion} states the following:
\begin{corollary}[Denise-Simion]
\label{C_DS}
Let $M(n)$ be the set of colored Motzkin paths which start at the origin
and end at the point $(n,0)$, that is, lattice paths whose steps are
$(+1,+1)$ (North-East), $(+1,-1)$ (South-East), or $(+1,0)$
(horizontal), and in which a horizontal step is colored red if it is at
zero abscissa and is colored either red or blue if it is at a positive
abscissa. To each such path $p$ we associate its number, $s(p)$, of
occurrences of two consecutive steps of the form $(NE, red)$ or
$(blue,red)$ or $(blue,SE)$ or $(NE,SE)$. Then 
$$
\sum_{p\in M(n)} t^{s(p)}=\sum_{j\geq 0}(-1)^j(1-t)^j\binom{n-j}{j}
C_{n-j}. 
$$
\end{corollary} 
As a consequence of equation (\ref{E_gdc}), the above cited result is
also a variant of the Gessel-Shapiro result, since it states 
\begin{equation}
\sum_{p\in M(n)} t^{s(p)}=g(L_n,t). 
\end{equation}
Corollary~\ref{C_DS} is a consequence of results on the polynomials 
$P_n(t)$, which may be defined by the relation
\begin{equation}
\label{E_gP}
g(L_n,t)=(1-t)^{n+1} +t\cdot P_{n+1}(t).  
\end{equation}
The polynomials $P_n(t)$ are related to counting noncrossing partitions, 
weighted according to the number of their {\em filler} points. Denise
and Simion~\cite[Definition 2.5]{Denise-Simion} define a filler of a
noncrossing partition $\pi$ as a point $i\in\{2,\ldots,n\}$
such that either $i-1$ and $i$ are in the same block and $i$ is the
largest element in its block or $i$ forms a singleton block and $i-1$ is
not the largest element of its block. Introducing $m(\pi)$ for the
number of filler points of the noncrossing partition $\pi$,
according to the proof of
\cite[Lemma 3.3]{Denise-Simion}, the polynomial $P_k(t)$ is the total
weight of all noncrossing partitions $\pi$ of $\{1,\ldots,k-1\}$ where 
the weight of $\pi$ is 
$$
w(\pi)=
\left\{
\begin{array}{lr}
\frac{1-(1-t)^k}{t}&\mbox{if $\pi=1/2/\ldots/k-1$},\\
t^{m(\pi)-1}&\mbox{otherwise}.\\
\end{array}
\right.
$$
The proof of this fact is very similar to our proof of
our Lemma~\ref{L_GS}. It begins with stating that one may start with an
arbitrary noncrossing partition o $k-j-1$ elements (counted by a Catalan
number) and then insert $j$ additional points in any of
$\binom{k-j-1}{j}$ ways. The $j$ additional points inserted will be the
(marked) filler points, the rest of the proof is slightly different only
because a different weight function is used. That said, from a
combinatorial perspective, inserting a filler after a prescribed
element in a non-crossing partition is an operation that behaves exactly
the same way as attaching a first-born leaf to a nonroot vertex. Thus,
in analogy of Lemma~\ref{L_GS}, we have the following statement.
\begin{lemma}[Denise-Simion]
\label{L_GSf}
The coefficient of $x^k$ in $g(L_d,x)$ is the number of noncrossing
partitions on $\{1,2,\ldots,d\}$ with $k$ fillers.
\end{lemma} 
Lemma~\ref{L_GSf} appears in the implicit form of 
$\sum_{p\in M(n)} t^{s(p)}=\sum_{\pi\in NC(n)} t^{m(\pi)}$ in the proof
  of~\cite[Corollary 3.8]{Denise-Simion}.  It is also worth noting that 
the proof of~\cite[Remark 3.4]{Denise-Simion}, involving the description
of all non-crossing partitions with exactly one filler point
involuntarily outlines a ``blueprint'' for a new construction of a
simplicial complex whose $h$-polynomial is the toric $h$-polynomial of a
cube. By working out small examples it is easy to verify that this
simplicial complex is not isomorphic in general to the one introduced by
Billera, Chan, and Liu~\cite[Theorem 3.2]{Billera-Chan-Liu}. 

We may obtain a new recursion formula for the polynomials $g(L_n,t)$
by observing a consequence
of~\cite[Lemma 3.3]{Denise-Simion}. This states that the polynomials $P_k(t)$  
satisfy the following recursion. 
\begin{equation}
\label{E_Prec}
P_k(t)=(1-t)^{k-1}+\sum_{i=1}^{k-1} \left((1-t)^i P_{k-i}(t)+t
P_i(t)P_{k-i}(t)\right). 
\end{equation}
Using (\ref{E_gP}) we may show the following.
\begin{proposition}
Introducing $g_n(t):=g(L_n,t)$ we have the following recursion.
$$
g_k(t)=(1-t)^k+\sum_{i=1}^k g_{i-1}(t)\left(g_{k-i}(t)-(1-t)^{n-i}\right).
$$
\end{proposition}
\begin{proof}
(\ref{E_gP}) implies 
$$
P_n(t)=\frac{g_{n-1}(t)-(1-t)^n}{t}.
$$
Substituting  this into (\ref{E_Prec}) we obtain that 
$\frac{g_{k-1}(t)-(1-t)^k}{t}$ is the sum of $(1-t)^{k-1}$ and of 
$$\sum_{i=1}^{k-1} \left((1-t)^i
\frac{g_{k-i-1}(t)-(1-t)^{k-i}}{t}+(g_{i-1}(t)-(1-t)^i)\frac{g_{k-i-1}(t)-(1-t)^{k-i}}{t}\right).
$$
Multiplying both sides by $t$ and adding $(1-t)^{k}$ to both sides 
yields 
\begin{align*}
g_{k-1}(t)&=(1-t)^k+t (1-t)^{k-1}\\
&+
\sum_{i=1}^{k-1} \left((1-t)^i
(g_{k-i-1}(t)-(1-t)^{k-i})
+(g_{i-1}(t)-(1-t)^i)(g_{k-i-1}(t)-(1-t)^{k-i})\right)\\
&= (1-t)^k+t (1-t)^{k-1} +\sum_{i=1}^{k-1} (1-t)^i g_{k-i-1}(t)
-(k-1) (1-t)^{k}\\
&+\sum_{i=1}^{k-1} g_{i-1}g_{k-i-1}
-\sum_{i=1}^{k-1} (1-t)^i (g_{k-i-1}(t)- \sum_{i=1}^{k-1} g_{i-1}(t)
(1-t)^{k-i}+(k-1)(1-t)^k.\\ 
\end{align*}
Canceling the terms $\sum_{i=1}^{k-1} (1-t)^i g_{k-i-1}(t)$ and
$(k-1)(1-t)^k$ which also appear with a negative sign and shifting the
index $k$ up by one yields the stated identity. 
\end{proof}

We conclude this section by observing that Corollary~\ref{C_DS} links
the study of the cubical toric $h$-polynomials to the combinatorial
theory of orthogonal polynomials developed by
Viennot~\cite{Viennot}. This theory uses weighted Motzkin paths to find
the {\em moment functional} of an orthogonal polynomial sequence given by a
recursion formula. Recall that a moment functional is a linear map $f: {\mathbb
  K}[x]\rightarrow {\mathbb K}$ from a polynomial ring to its field of
scalars, and it is uniquely defined by the {\em moments} $f(x^n)$. A 
sequence of polynomials $\{p_n(x)\}_{n=0}^{\infty }$ is orthogonal with
respect to the moment functional $f$ if $f(p_k(x)\cdot p_l(x))=0$
whenever $k\neq l$.  The following statement is a consequence 
of~\cite[Proposition 17]{Viennot}.  
\begin{theorem}[Viennot]
If an orthogonal polynomial sequence
$\{p_n(x)\}_{0}^{\infty }$ is defined 
by the recursion formula 
$$
p_{n+1}(x)=(x-b_n)p_n(x)-\lambda_n p_{n-1}(x)\quad \mbox{for $n\geq 1$}, 
$$
subject to the initial conditions $p_0(x)=1$ and $p_1(x)=x-b_0$, then 
the sequence $\{p_n(x)\}_{n=0}^{\infty }$ is orthogonal with respect to
the moment functional $f: {\mathbb K}[x]\rightarrow {\mathbb K}$
where $f(x^n)$ is the total weight of all Motzkin paths from $(0,0)$ to
$(n,0)$ such that each $NE$ step has weight $1$, each horizontal step 
at abscissa $i$ has weight $b_i$ and each $SE$ step starting at abscissa 
$i$ has weight $\lambda_i$. 
\end{theorem}
Thus Corollary~\ref{C_DS} suggests that the polynomials $g(L_n,x)$
are related to the orthogonal polynomial sequence 
$\{p_n(x)\}_{n=0}^{\infty }$, defined 
by the recursion formula 
$$
p_{n+1}(x)=(x-2)p_n(x)-p_{n-1}(x)\quad \mbox{for $n\geq 1$}, 
$$
subject to the initial conditions $p_0(x)=1$ and $p_1(x)=x-1$. 
Introducing 
$$b_n(x):=(-1)^n p_n(-x),$$ 
we obtain the recursion formula
\begin{align*}
b_{n+1}(x)
&=(-1)^{n+1}p_{n+1}(x)
=(-1)^{n+1}\left((-x-2)p_n(-x)-p_{n-1}(-x)\right)\\
&=(x+2)b_{n}(x)-b_{n-1}(x)\quad \mbox{for $n\geq 1$},
\end{align*}
subject to the initial conditions $b_0(x)=1$ and $b_1(x)=x+1$. 
These are exactly the {\em Morgan-Voyce polynomials} which are widely
studied. They first appeared in the study of electrical
networks~\cite{Morgan-Voyce}, some of the other early references include
Swamy's work~\cite{Swamy1} and~\cite{Swamy2}, the latest publication on them
appeared in 2008~\cite{Stoll-Tichy}. By Corollary~\ref{C_DS}, the
linear map defined by $x^n \mapsto g(L_n,x)$ may be considered as a
polynomial generalization of the moment functional of the Morgan-Voyce
polynomials.   

\section{Concluding remarks}
\label{s_c}

While giving a few modest answers, this paper raises many questions. 
The first ones are inspired by the results of Section~\ref{s_reflect}. What we
can observe, that by applying a sequence of involutions (taking mirror
images of  appropriate objects after and before applying a less trivial
bijection and its inverse) we obtain a different model, in which the
the definitions of the special vertices become overall a little less
technical. There are many other ways to transform Chan's model, but when
we take a ``random'' transformation, we usually end up with highly
technical definitions for our special vertices. In this sense, the
transformation presented in  Section~\ref{s_reflect} appears to be a
``lucky guess'' (actually it was not, the new model was built directly,
inspired by the formulas in Section~\ref{s_expl}, the bijection was found
later). Is there a way to recode the original
model into an even simpler one? Or is it better to start with a
completely different model, inspired by some other form of 
the formulas in Section~\ref{s_expl}?
At a more philosophical level, it seems that ``easy''
combinatorial transformations take one statistics of plane trees into
another statistic of plane trees. Analogous results for permutation 
statistics have been produced for over a hundred year, perhaps it is
worthwhile to take a similar systematic approach to the study of various
statistics on plane trees. 

We have a strong reason to suspect that the variant of the Gessel-Shapiro
result that may be found in the work of Denise and Simion~\cite{Denise-Simion}
will not be shown equivalent via some bijection that is similar to the
one presented in Section~\ref{s_reflect}. As it was mentioned after
Lemma~\ref{L_GSf}, it is possible to construct a new simplicial 
complex verifying  the
Billera-Chan-Liu result~\cite[Theorem 3.2]{Billera-Chan-Liu} stating
that the toric $h$-polynomial of a cube is the $h$-polynomial of a
simplicial complex. However, the resulting simplicial complex will not be
isomorphic to the one that can be found 
in~\cite{Billera-Chan-Liu}. For the postorder tree model, obviously, one
ends up constructing essentially the same simplicial complex (associated
to the same noncrossing partitions). Thus it
may be worthwhile to generalize the Denise-Simion colored Motzkin path
enumeration problem for $f_d(0,0,x)$ to questions whose answers are
given by the polynomials $f_d(i,j,x)$. This seems feasible since, as it 
was noted above, essentially the same insertion technique may be used to
prove Lemma~\ref{L_GSf} as the one used to show Lemma~\ref{L_GS}.
One only needs to come up with the appropriate analog of type $(1)$ and 
type $(2)$ special elements such that the analogues of Lemma~\ref{L_t1}
and Lemma~\ref{L_t2} become valid. 

Obviously it would be great to use one of the new variants (presented or
to be worked out) to extend the results of Billera, Chan, and Liu~\cite{{Billera-Chan-Liu}} to all
shellable cubical complexes. This seems hard, because after careful reading 
of~\cite{Billera-Chan-Liu} we will discover that the construction does
not only depend on constructing shellable simplicial complexes
associated to cubical shelling components but one also changes the
underlying cubical complex. No further generalization seems possible
there without modifying or generalizing that rearranging of cubes. 

Finally the most interesting and enigmatic question is raised by the
connection between the polynomials $g(L_n,x)$ and the Morgan-Voyce 
polynomials. Is it a coincidence or is there a deeper reason? Is there a
way to relate other $g$-polynomials to orthogonal polynomials? Is there
a way to extend the correspondence to all polynomials $f_d(i,j,x)$? If
nothing else, these questions certainly make a future ``third look'' at
the toric $h$-polynomials of cubical complexes worthwhile.


\begin{thebibliography}{99}

\bibitem{Bayer-Ehrenborg}
Margaret M.\ Bayer and Richard Ehrenborg, 
The toric $h$-vectors of partially ordered sets,
{\em Trans.\ Amer.\ Math.\ Soc.\ } {\bf 352} (2000), 4515--4531. 

\bibitem{Billera-Brenti}
Louis J.\ Billera and Francesco Brenti, 
Quasisymmetric functions and Kazhdan-Lusztig polynomials, 
preprint 2007, to appear in Israel Journal of Mathematics, 
arXiv:0710.3965v2 [math.CO].

\bibitem{Billera-Chan-Liu}
Louis J.\ Billera, Clara S. Chan, and Niandong Liu,
Flag Complexes, Labelled Rooted Trees, and Star Shellings,
in: Contemporary Mathematics Vol. 223, ``Advances in Discrete and
Computational Geometry: Proceedings of the 1996 AMS-IMS-SIAM Joint
Summer Research Conference, Discrete and Computational Geometry--Ten
Years Later, July 14-18, 1996, Mount Holyoke College'' (Bernard
Chazelle, Jacob E. Goodman, Richard Pollack Eds.), AMS, Providence, RI, 1999.

\bibitem{Bruggeser-Mani} M.\ Bruggeser and P.\ Mani, Shellable
  decompositions of cells and spheres, {\it Math. Scand.} {\bf 29}
  (1971), 197-205.    

\bibitem{Chan}
Clara S.\ Chan, 
Plane trees and $H$-vectors of shellable cubical complexes.
SIAM J. Discrete Math. {\bf 4} (1991), 568--574. 

\bibitem{Denise-Simion}
Alain Denise and Rodica Simion, 
Two combinatorial statistics on Dyck paths,
{\em Discrete Math.} {\bf 137} (1995), 155--176. 

\bibitem{Ehrenborg-Hetyei}
R.\ Ehrenborg and G.\ Hetyei
Flags and shellings of Eulerian cubical posets,
{\em Ann.\ Comb.\ } {\bf 4} (2000), 199--226. 

\bibitem{Kreweras}
G.\ Kreweras, G.
Sur les partitions non crois\'ees d'un cycle, 
{\em Discrete Math.\ } {\bf 1} (1972), 333--350.

\bibitem{Morgan-Voyce}
A.\ M.\ Morgan-Voyce,  
Ladder Network Analysis Using Fibonacci Numbers,
{\em IRE Trans. Circuit Th.} CT-6, 321--322, Sep. 1959. 

\bibitem{Stanley-GH}
R.P.\  Stanley,
Generalized $H$-Vectors, Intersection Cohomology of Toric Varieties and
Related Results, in: ``Commutative Algebra and Combinatorics'' (M.\ Nagata
and H.\ Matsumura, eds.), Advanced Studies in Pure Math. 11, Kinokuniya,
Tokyo and North-Holland, Amsterdam, New York, Oxford, 1987, 187--213. 

\bibitem{Stanley-EC1}
R.\ P.\  Stanley,
``Enumerative Combinatorics, Volume I,'' Cambridge
University Press, Cambridge, 1997.

\bibitem{Stanley-EC2}
R.P.\  Stanley,
``Enumerative Combinatorics, Volume II,'' Cambridge
University Press, Cambridge, 1999.


\bibitem{Stoll-Tichy}
T.\ Stoll and R.\ Tichy,
Diophantine equations for Morgan-Voyce and other modified orthogonal
polynomials, {\em Math.\ Slovaca} {\bf 58} (2008), 11--18. 

\bibitem{Swamy1}
M.\ N.\ S.\ Swamy, 
Properties of the Polynomials Defined by Morgan-Voyce
{\em Fibonacci Quart.} {\bf 4} (1966), 73--81. 

\bibitem{Swamy2}
M.\ N.\ S.\ Swamy, 
Further properties of Morgan-Voyce polynomials,
{\em Fibonacci Quart.} {\bf 6} (1968), 167--175. 

\bibitem{Viennot}
X.\ Viennot, 
``Une th\'eorie combinatoire des polyn\^omes orthogonaux g\'en\'eraux,''\\
Lecture Notes LACIM, UQAM, 1983.\\ 
Available online at the author's
personal website {\tt http://web.mac.com/xgviennot}  

\bibitem{Yano-Yoshida}
F.\ Yano and H.\ Yoshida,
Some set partition statistics in non-crossing partitions and generating
functions, {\em Discrete Math.\ } {\bf 307} (2007), no. 24, 3147--3160. 

\end{thebibliography}
\end{document}